\documentclass[12pt]{amsart}
\usepackage{mathrsfs}
\usepackage{amsfonts, amsmath}
\usepackage{amsmath,amstext,amsthm,amssymb,amsxtra,verbatim}
\usepackage{enumerate}
\usepackage{graphicx}
\usepackage{siunitx}
\usepackage{tikz-cd}
\usepackage{color}
\usepackage[colorlinks,
linkcolor=black,
anchorcolor=black,
citecolor=black
]{hyperref}
\usetikzlibrary{arrows}
\numberwithin{equation}{section}

\usepackage{mathabx}

\usepackage{mathtools}
\usepackage[tableposition=top]{caption}
\usepackage{booktabs,dcolumn}

\DeclareFontFamily{OT1}{rsfs}{}
\DeclareFontShape{OT1}{rsfs}{n}{it}{<-> rsfs10}{}
\DeclareMathAlphabet{\mathscr}{OT1}{rsfs}{n}{it}

\theoremstyle{plain}

\newtheorem{theorem}{Theorem}[section]

\newtheorem{lemma}[theorem]{Lemma}
\newtheorem{corollary}[theorem]{Corollary}

\theoremstyle{definition}

\newcommand{\abs}[1]{\left\vert#1\right\vert}
\newcommand{\norm}[1]{\Vert#1\Vert}

\newcommand\R{\mathbb{R}}

\def \supp {\,{\rm supp}\,}

\begin{document}
	\title[Estimating oscillatory integrals using broad-narrow method]{A sharp decay estimate for degnerate oscillatory integral operators using broad-narrow method}
	
	\author{Shaozhen Xu}
	\address{School of Information Engineering, Nanjing Xiaozhuang University, China, 211171}
	\email{xushaozhen@njxzc.edu.cn}

	
	\begin{abstract}
	 	We use broad-narrow method to estabish the sharp $L^4$ decay estimate for a class of degenerate oscillatory integral operators in $(2+1)$ dimensions. Especially, the model phase function is 
	 	\[xt^2+y^2t,\]
	a cubic homogeneous polynomial which is degenerate in the sense of \cite{tang2006decay}.
	\end{abstract}
	\maketitle
	\tableofcontents
	
	\section{Introduction}
	It is always a mystery to determine the decay rate for integrals with integrands containing oscillatory elements. For the scalar oscillatory integral
	\begin{equation*}
		I(\lambda)=\int_{\R^d}e^{i\lambda S(x)}\psi(x)dx,
	\end{equation*}
	we are interested in determing the optimal decay rate whenever the phase function satisfies certain assumptions. Stationary phase method tells that if the critical points of $S(x)$ are nondegenerate, then we have the sharp decay estimate 
	\begin{equation*}
		I(\lambda)\leq C\lambda^{-\frac{d}{2}}.
	\end{equation*}
	It should be pointed out that the implicit constant $C$ depends on the support function $\psi$ even the phase function $S$. It strongly demonstrate the "local" characteristic constrast to the "global" van der Corput-type estimates. Readers may refer to \cite{basu2021stationary} and \cite{gilula2018some} for more comprehensive contents. It is natural that we can further consider the oscillatory integrals with degenerate phases, if the phase function is real-analytic, Arnold posed the hypothesis that the sharp decay rate is determined by the Newton distance of the phase. This was confirmed by \cite{varchenko1976newton} in which the phase functions are assumed to satisfy certain nonsingular conditions.
	
	As was generalized to the operator setting, we call the operators of the form 
	\begin{equation*}
		T_\lambda f(x):=\int_{\R^{d_Y}}e^{i\lambda S(x,y)}\psi(x,y)f(y)dy, \quad (x,y)\in \R^{d_X}\times\R^{d_Y}		
	\end{equation*}
as $(d_X+d_Y)-$\emph{dimensional} operators. If the phase function $S$ is nondegenerate in the sense that the Hessian does not vanish on the support of $\psi$, then H\"{o}rmander's lemma \cite{hormander1973oscillatory} gives the optimal decay estimate whenver $d_X=d_Y=d$. The model phase is $S(x,y)=x\cdot y$ which corresponds to Fourier transform. In fact, there is another geometric understanding for H\"{o}rmander's lemma. Based on the phase function, we can write the canonical relation 
\[C_S:=\{(x, \xi; y, \eta)\}=\{(x,\nabla_xS; y, -\nabla_yS)\}\subset T^{*}\R^d\times T^{*}\R^d,\]
and view this as a Lagrangian submanifold when endowed with a symplectic form $dx\wedge d\xi+dy\wedge d\eta$. Define the left and right projection mapping respectively as
\[\pi_L:C_S\longrightarrow (x,\xi), \quad \pi_R:C_S\longrightarrow (y,\eta).\] In the language of geometry, H\"{o}rmander's lemma says that if both $\pi_L$ and $\pi_L$ are local diffeomorphisms, then $\|T_\lambda\|_{L^2\to L^2}$ has the optimal decay estimate $\lambda^{-\frac{d}{2}}$. Obviously, the singular types of the mappings $\pi_L$ or $\pi_R$ influence the decay rates of $T_\lambda$. Further researches in this direction, see \cite{pan1990oscillatory} \cite{greenleaf1994fourier}\cite{greenleaf1998fourier}\cite{greenleaf1999oscillatory}\cite{greenleaf2002oscillatory}\cite{comech1997integral}\cite{comech2000integral}.

However, the application of this geometric method is limited by the requirement $d_X=d_Y$ and a frustrated fact is that even the phase function is a simple homogeneous polynomial, the singularities of the corresponding mappings $\pi_L$ and $\pi_R$ are complicated. One way out of the geometric constraints is to focus on the analytic properties, this leads to the thorough understanding of $(1+1)-$dimensional operators in the works on $L^2$ mapping properties \cite{phong1992oscillatory}\cite{phong1994models}\cite{phong1997newton}\cite{rychkov2001sharp}\cite{greenblatt2005sharp} and $L^p$ mapping properties \cite{phong2001multilinear}\cite{yang2004sharp}\cite{shi2017sharp}\cite{xiao2017endpoint}\cite{shi2019damping}. 

In $(2+1)-$dimensional case, Tang\cite{tang2006decay} obtained the (nearly) sharp $L^2$ decay estimates for operators with nondegenerate homogeneous polynomial phases. Greenleaf-Pramanik-Tang \cite{greenleaf2007oscillatory} extend them to more higher dimensional cases. Fairly speaking, the general higher dimensional cases are little understood. When $d_X\neq d_Y$, researches about $T_\lambda$ are largely motivated by Fourier restriction (or extension) phenomenon initially raised by Stein in 1960's. For instance, if the underlying geometric object is codimension-$1$ sphere or paraboloid, it can be generalized to H\"{o}rmander-type oscillatory integral operators, progresses can be found in \cite{hormander1973oscillatory}\cite{bourgain1991p}\cite{bourgain2011bounds}\cite{guth2019sharp}; if the underlying geometric object are space curves, for instance see \cite{drury1985restrictions}, the generalization to corresponding oscillatory integral operators can be found in \cite{bak2004estimates}\cite{bak2009restriction}. Since the the tools of restriction estimates are fruitful, then we aim to explore the possibility of applying some of these tools to establish sharp $L^2$ decay estimates for degenerate oscillatory integral operators. 

Here, we consider the following $(2+1)-$dimensional operators
\begin{equation*}
	T_\lambda^\psi f(x,y)=\int_{\R}e^{i\lambda(x^2t+yt^2)}\psi(x,y,t)f(t)dt.
\end{equation*}
Although the phase function is a simple cubic homogeneous polynomial, it does not saisfy the Carleson-Sj\"{o}lin condition and is also outsdie scope of \cite{tang2006decay}. We use the broad-narrow method of Bourgain-Guth \cite{bourgain2011bounds} and the classical bilinear method dating back to Fefferman \cite{fefferman1970inequalities} to give the following theorem.
\begin{theorem}\label{Main-Thm}
	For the operator $T_\lambda^\psi$, we have the sharp $L^4$ decay estimate 
	\begin{equation}
		\|T_\lambda^\psi f\|_{L^4(\R^2)}\lesssim C_\psi\lambda^{-\frac{3}{8}} \|f\|_{L^4(\R)}.
	\end{equation}
\end{theorem}
In fact, by simple linear transfomation, we can generalize this theorem to more general cases.
\begin{corollary}
	Consider the following operators
	\begin{equation*}
		\tilde{T}_\lambda^\psi f(x,y)=\int_{\R}e^{i\lambda[(ax+by)t^2+(cx+dy)^2t]}\psi(x,,y,t)f(t)dt,
	\end{equation*}
if the matrix 
$\left(
\begin{matrix}
	a & b\\
	c & d
\end{matrix}\right)$ is nonsigular, we also have the sharp $L^4$ decay estimate
	\begin{equation*}
	\|\tilde{T}_\lambda^\psi f\|_{L^4(\R^2)}\lesssim C_{\psi, a,b,c,d}\lambda^{-\frac{3}{8}} \|f\|_{L^4(\R)}.
\end{equation*}
\end{corollary}

{\bf Notation:} In this paper, $A\lesssim B$ means that there exists an absolute constant $C$ such that $A\leq CB$. 

 \section{Outline of the proof}
 The strategy of proving Theorem \ref{Main-Thm} is similar to what have been done to Fourier restriction estimates. We use broad-narrow method to divide $T_\lambda^\psi f(x,y)$ into broad part and narrow part. Then we establish a rescaling lemma to deal with the narrow part. By employing the classical result of Phong-Stein \cite{phong1994models} in $(1+1)-$dimensional case, we establish a bilinear estimate for the broad part. Last, we combine the argument above and derive an induction relation which leads to the final result.
 
Before the formal argument, we need some uniform constraints on the support function. Denote the class of functions $\mathcal{U}$ by
\begin{equation*}
\mathcal{U}:=\{\psi(x,y,t):\supp{\psi}\subset[-1,1]^3, \abs{\partial_t^j\psi}\leq 1, \quad j=1,2,3\}.
\end{equation*}
Based on this, we further write
\begin{equation*}
	Q_p(\lambda):=\sup_{\psi\in\mathcal{U}}\sup_{\|f\|_{L^p(\R)}\leq 1}{\|T_\lambda^\psi f\|_{L^p(\R^2)}}.
\end{equation*}
To prove Theorem \ref{Main-Thm}, it suffices to show the following result.
\begin{theorem}\label{Red-Thm}
	\begin{equation*}
		Q_4(\lambda)\lesssim \lambda^{-\frac{3}{8}}.
	\end{equation*}
\end{theorem}
This can be derived by the following iteration relation.
\begin{lemma}\label{Main-Lemm} For any large number $K(K>10^4)$, we have
	\begin{equation*}
		Q_4(\lambda)\lesssim K\lambda^{-\frac{3}{8}}+K^{-\frac{1}{2}}Q_4(\lambda/K).
	\end{equation*}
\end{lemma}
It should be pointed out, using induction to prove oscillatory estimates had appeared previously in \cite{bak2009restriction} or even earlier in \cite{nagel1978differentiation}.

\section{Optimality}
Now we give an example to show that the decay estimate in Theorem \ref{Main-Thm} is sharp.

Suppose $\psi$ is a nonnegative cut-off function and satisfies
\begin{equation}\label{Supp-Func}
	\psi(x,y,t)=\begin{cases}
		0, &\quad \abs{(x,y,t)}\geq 1,\\
		1, &\quad \abs{(x,y,t)}\leq \frac{1}{2}.
	\end{cases}
\end{equation}
We choose the test function as
\begin{equation}\label{Tes-Func}
	f(t)=\chi_{[0,1]}(t).
\end{equation}
Assume the priori estimate 
\begin{equation*}
	\|T_\lambda^\psi f\|_{L^4(\R^2)}\lesssim C_\psi\lambda^{-\delta} \|f\|_{L^4(\R)}.
\end{equation*}
Then for the specific function \eqref{Tes-Func}, we see that
\begin{equation*}
	\left[\iint\abs{\int_0^1e^{i\lambda(xt^2+y^2t)}\psi(x,y,t)dt}dxdy\right]^{\frac{1}{4}}\lesssim C_\psi\lambda^{-\delta}.
\end{equation*}
With the support function \eqref{Supp-Func}, we have
\begin{equation*}
	\lambda^{-\frac{3}{8}}\lesssim \left[\int_{\abs{y}\lesssim \lambda^{-1}}\int_{\abs{x}\lesssim \lambda^{-\frac{1}{2}}}\abs{\int_0^1e^{i\lambda(x^2t+yt^2)}\psi(x,y,t)dt}dxdy\right]^{\frac{1}{4}}\lesssim C_\psi\lambda^{-\delta}.
\end{equation*}
Since the inequality holds for arbitrarily large $\lambda$, then it requires
\begin{equation*}
	\delta\leq \frac{3}{8}.
\end{equation*}
This ultimately shows the optimality.
\section{Proof of Theorem \ref{Red-Thm}}
Broad-narrow method was introduced by Bourgain-Guth in \cite{bourgain2011bounds} which efficiently reduce the linear restriction estimates to multilinear estimates which are understood well. As usual, we decompose the test function into $K$ parts
\begin{equation*}
	f(t)=\sum_{j=0}^{K-1} f_j(t),
\end{equation*}
where $\supp f_j\subset [\frac{j}{K}, \frac{j+1}{K})\bigcup (-\frac{j+1}{K},-\frac{j}{K})]$. For a positive number $\alpha\in (0,1)$(in application, we choose $\alpha=10^{-4}$), we say a point $(x,y)\in [-1, 1]^2$ is $\alpha-$\emph{broad} if 
\begin{equation*}
	\max_j{\abs{T_\lambda^\psi f_j(x,y)}}\leq \alpha \abs{T_\lambda^\psi f(x,y)},
\end{equation*}
otherwise, the point $(x,y)$ is called \emph{narrow}. We write
\begin{equation*}
	Br_\alpha (T_\lambda^\psi f)(x,y)=\begin{cases}
		\abs{T_\lambda^\psi f(x,y)} \quad \text{if }(x,y) \text{ is $\alpha-$broad};\\
		0 \quad\quad\quad\quad\quad\  \text{if }(x,y) \text{ is narrow}.
	\end{cases}
\end{equation*}
Therefore, pointwisely we have
\begin{equation*}
	\abs{T_\lambda^\psi f(x,y)}\leq Br_\alpha (T_\lambda^\psi f)(x,y)+\alpha^{-1}\max_j\abs{T_\lambda^\psi f_j(x,y)}.
\end{equation*}
This implies the following inequality
\begin{equation}\label{Bro-Nar}
\iint \abs{T_\lambda^\psi f(x,y)}^4dxdy\leq \iint \abs{Br_\alpha (T_\lambda^\psi f)(x,y)}^4dxdy+\alpha^{-4}\sum_{j=1}^{k-1}\iint \abs{T_\lambda^\psi f_j(x,y)}^4dxdy.
\end{equation}
We deal with the first term using bilinear estimates and the latter one using a rescaling estimate we now turn to.
\subsection{Rescaling argument}
This part is devoted to proving a degenerate rescaling estimate which is basically same with the parabolic rescaling estimate. Now we state the main result.
\begin{lemma}
	\begin{equation}\label{Res-Est}
		\sup_{\psi\in\mathcal{U}}\norm{T_\lambda^\psi f_j}_{L^4(\R^2)}\lesssim K^{-\frac{1}{2}}Q_4(\lambda/K)\norm{f_j}_{L^4(\R)}.
	\end{equation}
\end{lemma}
\begin{proof}
For a cut-off function $\phi$, we may assume $f_j(t)=f(t)\phi_j(t)$ where $\phi_j(t)=\phi\left(K\left(t-\frac{j}{K}\right)\right)$, then
	\begin{align*}
		T_\lambda^\psi f_j(x,y)&=\int_{\R}e^{i\lambda(x^2t+yt^2)}\psi(x,y,t)\phi_j(t)f(t)dt,\\
		&=\int_{\R}e^{i\lambda(x^2t+yt^2)}\psi(x,y,t)\phi\left(K\left(t-\frac{j}{K}\right)\right)f(t)dt\\
		&=\int_{\R}e^{i\lambda\left[x^2\frac{s+j}{K}+y(\frac{s+j}{K})^2\right]}\psi\left(x,y,\frac{s+j}{K}\right)\phi(s)f\left(\frac{s+j}{K}\right)\frac{ds}{K},\\
		&=e^{i\frac{\lambda}{K}\left(x^2j+\frac{yj^2}{K}\right)}\int_{\R}e^{i\frac{\lambda}{K}\left(x^2s+\frac{ys^2}{K}\right)}e^{i\frac{2\lambda jsy}{K^2}}\psi\left(x,y,\frac{s+j}{K}\right)\phi(s)f\left(\frac{s+j}{K}\right)\frac{ds}{K}.
	\end{align*}
Set $\psi_j(x,y,s)=e^{i\frac{2\lambda jsy}{K^2}}\psi\left(x,y,\frac{s+j}{K}\right)$ and $f_{j,K}(s)=\phi(s)f\left(\frac{s+j}{K}\right)$ it is easy to verify that $\psi_j\in \mathcal{U}$. Therefore,
\begin{equation*}
	\abs{T_\lambda^\psi f_j(x,y)}=\frac{1}{K}\abs{T_{\lambda/K}^{\psi_j}f_{j,K}(x,y/K)}.
\end{equation*}
So we have
\begin{align*}
	\norm{T_\lambda^\psi f_j}_{L^4(\R^2)}^4&=\frac{1}{K^4}\cdot K\norm{T_{\lambda/K}^{\psi_j}f_{j,K}}_{L^4(\R^2)}^4\\
	&\leq \frac{1}{K^3}Q_4^4\left(\frac{\lambda}{K}\right)K\norm{f_{j,K}}_{L^4(\R^2)}^4\\
	&=\frac{1}{K^2}Q_4^4\left(\frac{\lambda}{K}\right)\norm{f_j}_{L^4(\R^2)}^4.
\end{align*}
This implies the inequality \eqref{Res-Est}.
\end{proof}
Consequently, the latter term in RHS of \eqref{Bro-Nar} is bounded from above by
\begin{equation}\label{Nar-Est}
	10^{4}K^{-2}Q_4^4\left(\lambda/K\right)\norm{f}_{L^4(\R)}^4.
\end{equation}
\subsection{Bilinear estimate}
We now deal with the first term in RHS of \eqref{Bro-Nar}. With the assumption $\alpha=10^{-4}$, we know that for each $\alpha-$broad point $(x,y)$ there exist $j,k$ with $|j-k|\geq 2$ such that 
\begin{equation*}
	\abs{T_\lambda^\psi f(x,y)}\leq K\abs{T_\lambda^\psi f_j(x,y)}^{\frac{1}{2}}\abs{T_\lambda^\psi f_k(x,y)}^{\frac{1}{2}}.
\end{equation*}
Notice that the indices $j,k$ depend on the point $(x,y)$,  we use summation over all indices to eliminate such dependence. For each $\alpha-$broad point $(x,y)$ we always have
\begin{equation*}
		\abs{T_\lambda^\psi f(x,y)}^4\leq K^4\sum_{\abs{j-k}\geq 2}\abs{T_\lambda^\psi f_j(x,y)}^{2}\abs{T_\lambda^\psi f_k(x,y)}^{2}.
\end{equation*} 
This implies
\begin{align*}
	&\iint \abs{Br_\alpha (T_\lambda^\psi f)(x,y)}^4dxdy\\
	&\leq K^4\sum_{\abs{j-k}\geq 2}\iint\abs{T_\lambda^\psi f_j(x,y)}^{2}\abs{T_\lambda^\psi f_k(x,y)}^{2}dxdy\\
	&=K^4\sum_{\abs{j-k}\geq 2}\iint\abs{T_\lambda^\psi f_j(x,y)T_\lambda^\psi f_k(x,y)}^{2}dxdy\\
	&=K^4\sum_{\abs{j-k}\geq 2}\iint\abs{\iint e^{i\lambda\left(x^2t+yt^2+x^2s+ys^2\right)}\psi(x,y,t)\psi(x,y,s)f_j(t)f_k(s)dtds}^2dxdy.
\end{align*}
By changing variables
\begin{equation*}
	u=t+s,\qquad v=t^2+s^2,
\end{equation*}
then
\begin{align*}
	&\iint e^{i\lambda\left(x^2t+yt^2+x^2s+ys^2\right)}\psi(x,y,t)\psi(x,y,s)f_j(t)f_k(s)dtds\\
	=&\iint e^{i\lambda\left(x^2u+yv\right)}\psi(x,y,t(u,v))\psi(x,y,s(u,v))f_j(t(u,v))f_k(s(u,v))\frac{dudv}{2\abs{t-s}}.
\end{align*}
Write 
\begin{equation*}
	F_{j,k}(u,v)=\frac{f_j(t(u,v))f_k(s(u,v))}{2\abs{t-s}},\quad \psi(x,y,u,v)=\psi(x,y,t(u,v))\psi(x,y,s(u,v)),
\end{equation*}
and transform the integral above to
\begin{equation*}
	\iint e^{i\lambda\left(x^2u+yv\right)}\psi(x,y,u,v)F_{j,k}(u,v)dudv.
\end{equation*}
Return to the broad part
\begin{align*}
	&\iint \abs{Br_\alpha (T_\lambda^\psi f)(x,y)}^4dxdy\\
	&\leq K^4\sum_{\abs{j-k}\geq 2}\iint\abs{\iint e^{i\lambda\left(x^2u+yv\right)}\psi(x,y,u,v)F_{j,k}(u,v)dudv}^2dxdy.
\end{align*}
Observe that the phase function in the integrand can be viewed as separation of variables. Iterating the $(1+1)-$dimensional result of Phong-Stein \cite{phong1994models}, we arrive at
\begin{equation}\label{Dec-Est}
	\iint \abs{Br_\alpha (T_\lambda^\psi f)(x,y)}^4dxdy\leq C\lambda^{-\frac{3}{2}}K^4\sum_{\abs{j-k}\geq 2}\norm{F_{j,k}(u,v)}_{L^2(dudv)}^2.
\end{equation}
It should be noted that the constant $C$ depends on the upper bounds of $\abs{\partial_u^j\psi(x,y,u,v)}$ and $\abs{\partial_v^j\psi(x,y,u,v)}$ for $j=1,2$, by omitting some cumbersome details we can verify that 
\begin{equation*}
	\abs{\partial_u^j\psi(x,y,u,v)}\lesssim 1, \quad \abs{\partial_v^j\psi(x,y,u,v)}\lesssim 1,
\end{equation*}
and this is why the class $\mathcal{U}$ needs third derivatives. Since the supports of $f_j$ and $f_k$ are essentially separated, then 
\begin{align*}
	\norm{F(u,v)}_{L^2(dudv)}^2&=\iint\abs{\frac{f_j(t(u,v))f_k(s(u,v))}{2(t(u,v)-s(u,v))}}^2dudv\\
	&=\iint\frac{\abs{f_j(t)f_k(s)}^2}{\abs{2(t-s)}}dtds\\
	&\lesssim \frac{K}{\abs{j-k}}\norm{f_j}_{L^2}^2\norm{f_k}_{L^2}^2.
\end{align*} 
So return to \eqref{Dec-Est}, we know
\begin{align}
		\notag\sup_{\psi\in\mathcal{U}}\iint \abs{Br_\alpha (T_\lambda^\psi f)(x,y)}^4dxdy&\lesssim \lambda^{-\frac{3}{2}}K^4\sum_{\abs{j-k}\geq 2}\frac{K}{\abs{j-k}}\norm{f_j}_{L^2}^2\norm{f_k}_{L^2}^2\\ \notag
		&\lesssim  \lambda^{-\frac{3}{2}}K^4\sum_{\abs{j-k}\geq 2}\frac{1}{\abs{j-k}}\norm{f_j}_{L^4}^2\norm{f_k}_{L^4}^2\\\notag
		&\leq \lambda^{-\frac{3}{2}}K^4\left(\sum_{\abs{j-k}\geq 2}\frac{1}{\abs{j-k}^2}\right)^{\frac{1}{2}}\left(\sum_{\abs{j-k}\geq 2}\norm{f_j}_{L^4}^4\norm{f_k}_{L^4}^4\right)^{\frac{1}{2}}\\
		&\lesssim \lambda^{-\frac{3}{2}}K^4\norm{f}_{L^4}^4.\label{Bro-Est}
\end{align}
Thus, \eqref{Bro-Est} together with \eqref{Nar-Est} implies
\begin{equation*}
		Q_4^4(\lambda)\lesssim K^4\lambda^{-\frac{3}{2}}+K^{-2}Q_4^4(\lambda/K).
\end{equation*}
This leads to Lemma \ref{Main-Lemm}. Thus we complete the proof of Theorem \ref{Main-Thm}.

\section{Acknowlegement}
The author thanks Prof. Xiaochun Li for sharing the idea of using induction machinery to deal with degenerate oscillatory integral operators when the author visited University of Illinois at Urbana-Champaign during 2017-2018. The author would like to acknowledge financial support from Jiangsu Natural Science Foundation, Grant No. BK20200308. The author shows great gratitude to Ran Xu, Xiaonan Hu and Jie Tang for their constant encouragement.


\end{document}